\documentclass{article}
\usepackage{graphicx}
\usepackage{amsthm,amsfonts,amssymb}
\usepackage{amsmath}
\usepackage{hyperref}
\usepackage[latin1]{inputenc}
\usepackage{latexsym}

\newtheorem{theorem}{Theorem}
\newtheorem{lemma}{Lemma}
\newtheorem{proposition}{Proposition}
\newtheorem{definition}{Definition}
\newcommand \be{\begin{eqnarray*}}
\newcommand \ee{\end{eqnarray*}}
\newcommand \ben{\begin{eqnarray}}
\newcommand \een{\end{eqnarray}}

\def \build#1#2#3{\mathrel{\mathop{\kern 0pt#1}\limits_{#2}^{#3}}}

\def \build#1#2#3{\mathrel{\mathop{\kern 0pt#1}\limits_{#2}^{#3}}}

  \newdimen\AAdi%
\newbox\AAbo%
\def\AArm{\fam0 \rm}%
\def\AAk#1#2{\setbox\AAbo=\hbox{#2}\AAdi=\wd\AAbo\kern#1\AAdi{}}%
\def\BBone{{\AArm 1\AAk{-.8}{I}I}}%

\title{Asymptotic cost of cutting down random free trees}
\begin{document}
%\begin{frontmatter}
\markboth{ ~~\hrulefill~~E. Zohoorian Azad } {Cost of cutting down
random trees ~~\hrulefill~~ }

\author{Elahe Zohoorian Azad \\ Damghan
university of Iran }
%\address{{School of mathematics and computer sciences, Damghan
%university, Iran, p.o.Box 3671641167}}

%\email{zohorian@du.ac.ir}

\maketitle

\begin{abstract} In this work, we calculate
the limit distribution of the total cost incurred by splitting a
tree uniformly distributed on the set of all finite free trees,
appears as an additive functional induced by a toll equal to the
square of the size of tree. The main tools used are the recent
results connecting the asymptotics of generating functions with
the asymptotics of their Hadamard product, and the method of
moments.
\end{abstract}

%\begin{keyword}
 \emph{keywords.} Additive functional, Generating functions, Limit law,
 Recurrence.
%\end{keyword}

%\begin{AMS}
%68P10, 60C05, 60J65, 68R05
%\end{AMS}
%\end{frontmatter}

%%%%%%%%%%%%%%%%%%%%%%%%%%%%%%%%%%%%%%%%%%%%%%%%%%%%%%%%%%%%%%%%%

\section{Introduction}

%%%%%%%%%%%%%%%%%%%%%%%%%%%%%%%%%%%%%%%%%%%%%%%%%%%%%%%%%
Trees are structures suitable for data storage and supporting
computer algorithms, two fundamental aspects of data processing,
with applications in many fields. The cost of
``divide-and-conquer'' algorithms can be represented as an
additive functional of trees. While there are many studies on
additive functional (see, for example,
\cite{Mahmoud,Hwang,Rosler}), not enough attention has been given
to the distributions of functional defined on trees under the
uniform model. However, a main motivation for undertaking this
investigation is that it is key to analyzing a special type of a
Drop-Push model of percolation and coagulation(see \cite{EZoh}).\\

In this paper, we consider the additive functional defined on the
trees uniformly selected from the set of all the free trees of
size $n$, for $n$ given (called \emph{Cayley} trees in
\cite{Cay}), induced by the toll sequence $(n^2)_{n\geq 0}$ (see
definition of the Section \ref{cayrec}). Our main result, Theorem
\ref{Xn}, provides the limit distribution for a suitably
normalized version of this functional.

\begin{theorem}
\label{Xn} Let $X_n$ be the additive functional defined on the
uniform free trees of size $n$, induced by the toll $(n^2)_{n\geq
0}$. Then,
$$n^{-5/2}\ X_n \build{\longrightarrow }{}{\mathcal L}\ {\sqrt 2}\ \xi,$$
where $\xi$ is a random variable whose distribution is
characterized by its moments.:
$$\mathbb{E}(\xi^k)=\frac{k!\sqrt\pi}{2^{(7k-2)/2}\Gamma(\frac{5k-1}{2})}{\bar a}_k,$$
where
$${\bar a}_k=2(5k-6)(5k-4){\bar a}_{k-1}+\sum_{j=1}^{k-1}{{\bar a}_j {\bar a}_{k-j}}
\hspace{0,5cm} k\geq 2;\hspace{0,5cm}{\bar a}_1= \sqrt 2.$$
\end{theorem}

Curiously, the moments of our limit distribution are proportional
to the moments of the distribution of the average of the minimum
of a normalized Brownian Excursion, obtained by \cite[Theorem
3.3]{Jan}.

In what follows, $e=\left(e(t)\right)_{0\leq t\leq 1}$
 indicates a normalized Brownian Excursion.
\begin{theorem}
\label{Janson} The moments of the random variable $\eta$, defined
by
$$\eta = 4 \int\int_{ 0<s<t<1} \min_{ s\leq u\leq t} e(u) ds dt,$$
are given by the formula
$$\mathbb
E(\eta^k)=\frac{k!\sqrt\pi}{2^{(7k-4)/2}\Gamma(\frac{5k-1}{2})}\
{\omega}_k,$$ where
$${\omega}_k=2(5k-6)(5k-4){\omega}_{k-1}+\sum_{j=1}^{k-1}{{\omega}_j {\omega}_{k-j}}
\hspace{0,5cm} k\geq 2;\hspace{0,5cm}{\omega}_1=1.$$
\end{theorem}
It is not unusual, in this kind of problem to have more than one
characterization of a limit distribution. For instance, the
\emph{Wiener index }of certain trees is given by its moments
involving \emph{Airy} functions, and is alternatively
characterized in terms of a Brownian.\\

For the demonstration of Theorem \ref{Xn}, we apply the strategy
used in \cite{FK} to obtain the limiting distributions of the
additive functionals defined on Catalan trees, in particular the
singularity analysis of the generating series \cite{FO}. Indeed,
the Hadamard products appear naturally when one analyzes the
moments of additive functionals of trees. Theorem \ref{Xn} extends
to the moments of all order, although the analysis of asymptotic
behavior of the first moment, was made already in \cite{FFK}. The
steps taken here allow a rather mechanical calculation of
asymptotic moments of each order, thus facilitating the
application of the method of moments.

%%%%%%%%%%%%%%%%%%%%%%%%%%%%%%%%%%%%%%%%%%%%%%%%%%%%
\section{Generating functions}
\label{cayrec}
%%%%%%%%%%%%%%%%%%%%%%%%%%%%%%%%%%%%%%%%%%%%%%%%%%%%%%
We first establish here some notation. Let $T$ be a binary tree
and let $|T|$ denote the number of nodes in $T$. Suppose moreover
that $L (T) $ and $R (T) $ indicate, respectively, the left and
right subtrees rooted at the 2 children of the root of $T$. When
the tree is not binary, one can still have two subtrees $L(T)$ and
$R(T)$, by cutting an edge which can be considered as root.

\begin{definition}
 A functional $f$ defined on a binary tree is called an \emph{additive functional}
if it satisfies the recurrence $$f(T ) = f(L(T )) + f(R(T )) +
b_{|T|},$$ for any tree $T$ with $|T|\geq 1$. Here $(b_n)_{n \geq
1}$ is a given sequence, henceforth called the toll function.
\end{definition}

We analyze here a special additive functional on the trees,
uniformly distributed on $\{T:|T|=n\}$, for $n$ given . By a
result attributed to Cayley \cite{Cay}, there are $U_n = n^{n-2}$
\emph{free} trees ($U_{n}$ connected acyclic labelled graphs) on
$n$ nodes and accordingly, there are $T_n = n^{n-1}$ rooted trees
(in which a labelled node, is called \emph{root} of tree).
Consider the model in which initially each free tree of size $n$
is taken uniformly at random. Choose an edge at random among the
$n - 1$ edges of the tree, orient it in a random way, then cut it.
This separates the tree into an ordered pair of smaller trees,
that are now rooted; we call them the left and right subtrees.
Continue the process with each of the resulting subtree,
discarding the root. Assume\footnote{One can see \cite[Proposition
1]{EZoh} for the main motivation of giving this assumption.
Briefly, \cite{EZoh} analyzes a Drop-Push model of coagulation in
which particles are dropped onto a one dimensional lattice and
carry out a random walk until they encounter an empty site where
they become stuck. In such a model, the movements of the
particles, on the lattice, form an additive coalescence processes
which gives the good algorithmic reasons for considering the
recurrence (\ref{rec1}). In fact, in the Drop-Push model, the cost
of coalescence of two clusters of particles, at the dropping
moment of a particle, is given as the number of steps of the
particle until it sticks in an empty site and it is proven,
\cite[relation (8)]{EZoh}, that the expected cost of coalescence
of two clusters is proportional to the square of the length of the
cluster on which a particle drops.} that the cost incurred by
selecting the edge and splitting the tree in a tree of size $n$ is
$n^2$. Then $X_{n}$, the total cost incurred for splitting a
random tree of size $n$, satisfies, for $n \geq 1$, the recurrence

\begin{equation}
\label{rec1}X_{n} = X_{L_{n}} + X_{R_{n}} + n^2,
\end{equation}
where the indexes $L_{n}$ and $R_{n}$ are, respectively, the sizes
of left and right subtrees, obtained by division of the initial
tree of size $n$. So $X_{n}$ appears as the additive functional
induced by the toll sequence $(n^2)_{n\geq 1}$.\\

A motivation, coming from the analysis of algorithms, is as
follows. If time is reversed, this model described the evolution
of a random graph, from a graph completely disconnected to a tree
and which was used to analyze of the \emph{union-find} algorithms
\cite{CLR, Se, VF}. Knuth and Sch{\"o}nhage provided a first
analysis of it in 1978 (\cite{KS}), for different tolls
however.\\

Let $p_{n,k}$ be the probability for a tree of size $n$ to have
the left and right subtrees respectively of sizes $k$ and $n-k$.
Then
\begin{equation}
\label{probcond} p_{n,k}\ =\ {{n}\choose{k}}\ \frac{k^{k-1}
(n-k)^{n-k-1}}{2(n-1)n^{n-2}}.
\end{equation}
The binomial coefficient ${{n}\choose{k}}$ takes into account the
labelling of the left and right subtrees, and the quantity
$k^{k-1} (n-k)^{n-k-1}$ is the number of rooted trees of sizes $k$
and $n-k$. In the denominator, $n^{n-2}$ is the number of free
trees, $n-1$ is the number of the edges of the initial tree, and
finally the coefficient $2$ corresponds to the random orientation
of the selected edge. It is convenient to write this probability
in the form:
$$p_{n,k}\ =\ {\frac{n}{2(n-1)}\ \frac{c_k
c_{n-k}}{c_n}},$$
where, $\forall k \geq 1$, $$c_k\ =\ \frac{k^{k-1}}{k!}.$$\\

Let us start with the average of the cost function, $a_n:=\mathbb
E(X_n)$, $n \geq 1$, which is obtained recursively by conditioning
on the size of $L_{n}$:
\begin{eqnarray*}
a_n&=&\mathbb E\left[\mathbb E_L(X_L+X_{n-L}+n^2)\right]
\\
&=&\mathbb E_L (a_L+a_{n-L})+n^2
\\
&=&\sum_{j=1}^{n-1}p_{n,j}(a_j +a_{n-j})+n^2
\\
&=& \sum_{j=1}^{n-1}{\frac{n}{2(n-1)}\frac{c_j
c_{n-j}}{c_n}}(a_j+a_{n-j})+n^2.
\end{eqnarray*}
This recurrence can be rewritten as
\begin{equation}
\label{rec2}\frac{n-1}{n}c_n a_n =\sum_{j=1}^{n-1}{c_j a_j
c_{n-j}}+\frac{n-1}{n}c_n b_n,
\end{equation}
where $b_n=n^2$.\\

\emph{Remark.} We replaced $n^2$ by $b_n$, distinguishing the
general form of the generating function, so that one can always
consider
any toll function in the place of $n^2$.\\

\begin{definition}  The Hadamard product of two entire series
$F(z)=\sum_n f_n z^n$ and $G(z)=\sum_n g_n z^n$, denoted
$F(z)\odot G(z)$, is the entire series defined by
 $$(F \odot G)(z)\equiv F(z) \odot G(z):=\sum_n f_n g_n z^n.$$
\end{definition}

Multiplying the equality (\ref{rec2}) by $z^n/e^n$ and summing
over $n \geq 1$, we get
\begin{eqnarray}
A(z)&\odot&C(z/e)-\int_0^z \sum_{n} a_n c_n \frac{\omega^n}{e^n}
\frac{d\omega}{\omega}
\\
{{\lefteqn{\label{difren}=(A(z)\odot C(z/e))C(z/e)+
\sum_{n}\frac{n-1}{n} c_n b_n \frac{z^n}{e^n},}}}
\end{eqnarray}
 where $A(z)$ and
$C(z)$ denote the ordinary generating function of $(a_n)_{n\geq
1}$ and
$(c_n)_{n\geq 1}$, respectively.\\

In view of a result of Knuth and Pittel, \cite{KP}, we know the
singular expansion at the dominant singularity $z = e^{-1}$ of
$C(z)$:
\begin{equation}
\label{caygene}C(z) = 1 - \sqrt 2(1 - ez)^{1/2} + O(|1 - ez|).
\end{equation}
Moreover $C$ satisfies the functional relation $C(z) =z e^{C(z)}$.
\\

By differentiation, the relation (\ref{difren}) transforms into a
linear differential equation of the first order, which can be
readily solved by the variation-of-constants method. Briefly,
putting $f(z):=A(z)\odot C(z/e)$ and
$t(z):=\sum_{n}\frac{n-1}{n}c_n b_ne^{-n} z^n$, the relation
(\ref{difren}) takes the form
\begin{equation}
\label{differ}\int_0^z f(\omega)
\frac{d\omega}{\omega}=f(z)(1-C(z/e))-t(z).
\end{equation}
By taking derivatives, we obtain
$$\frac{df(z)}{dz}+f(z){\left(\frac{-1/z-\frac{dC(z/e)}{dz}}
{1-C(z/e)}\right)}={\left(\frac{1}{1-C(z/e)}\right)\frac{dt(z)}{dz}}.$$
On the other hand, the equality $C(z/e)=\frac{z}{e}e^{C(z/e)}$
implies
$$\frac{dC(z/e)}{dz}={C(z/e)}\left(\frac1z+\frac{dC(z/e)}{dz}\right).$$
Assuming now (without loss of generality) the initial condition
$a_1 c_1 =b_1= 0$, the solution found will be in the form
\begin{equation}
\label{recHada}A(z)\odot C(z/e)=\frac{C(z/e)}{1-C(z/e)}\int_0^z
\partial_\omega(\sum_{n}\frac{n-1}{n} c_n b_n \frac{\omega
^n}{e^n})\frac{d\omega}{C(\omega/e)}.
\end{equation}
And finally as $\frac{n-1}{n} c_n=\sum_{j=1}^{n-1}\frac{1}{2}c_j
c_{n-j}$, we have
\begin{equation}
\label{recHada1}A(z)\odot
C(z/e)=\frac{1}{2}\frac{C(z/e)}{1-C(z/e)}\int_0^z
\partial_\omega[B(\omega)\odot
C(\omega/e)^2]\frac{d\omega}{C(\omega/e)},
\end{equation}
where $B(\omega)$ denote the ordinary generating function of
$(b_n)_{n\geq 1}$.

%%%%%%%%%%%%%%%%%%%%%%%%%%%%%%%%%%%%%%%%%%%%%%%%%%%%
\section{Moments by singularity analysis}
\label{asmean1}
%%%%%%%%%%%%%%%%%%%%%%%%%%%%%%%%%%%%%%%%%%%%%%%%%%%%%%

Thanks to the singularity analysis technique, we can derive the
asymptotics of moments of each order. The singularity analysis is
a systematic \emph{complex-analytic} technique that relates the
asymptotic behavior of sequences to the behavior of their
generating functions in the proximity of their singularities. The
applicability of singular analysis rests on a technical condition:
the $\Delta$-\emph{regularity}. See \cite{FFK,FO}
for more details.\\

\begin{definition}
\label{delta-r} A function defined by a Taylor series about the
origin with radius of convergence equal to 1 is $\Delta$-{regular}
if it can be analytically continued in a domain of form
$$\Delta(\phi,\eta):=\{z:|z|<1+\eta,|\arg(z-1)|>\phi\},$$
for some $\eta>0$ and $0<\phi<\pi/2$. A function $f$ is said to
admit a singular expansion at $z = 1$, if it is $\Delta$-regular
and if one can find a sequence of complex numbers $(c_j)_{0\leq
j\leq J}$, and an increasing sequence of real numbers
$(\alpha_j)_{0\leq j\leq J}$, satisfying $\alpha_j <A$, where $A$
is a real number, such that the relation
$$
f(z)= \sum_{j=0}^J c_j (1-z)^{\alpha_j} + O(|1-z|^A)
$$
holds uniformly in $z\in\Delta(\phi,\eta)$. It is said to satisfy
a singular expansion with \emph{logarithmic terms} if,
$$f(z)=\sum_{j=0}^J c_j (L(z))(1-z)^{\alpha_j}
+ O(|1-z|^A),\hspace{.3cm}L(z):=\log \frac{1}{1-z},$$ where each
$c_j (.)$ is a polynomial.
\end{definition}

Recall the definition of the generalized polylogarithm:

\begin{definition}
\label{Li}
 For $\alpha$ an arbitrary complex number and $r$ a nonnegative integer,
the generalized polylogarithm function $Li_{\alpha,r}$ is defined
for $|z|< 1$, by
$$Li_{\alpha,r}(z):=\sum_{n\geq 1}\frac{(\log
n)^r}{n^{\alpha}}z^n.$$
\end{definition}
In particular, $Li_{1,0}(z)=L(z)$. Moreover, a useful property of
generalized polylogarithm functions is
$$Li_{\alpha,r}\odot Li_{\beta,s}=Li_{\alpha+\beta,r+s}.$$

The singular expansion of the polylogarithm involves the Riemann
zeta function (see for example \cite[Theorem 4]{FFK}).

\begin{lemma}
\label{expan Li} The function $Li_{\alpha,r}(z)$ is
$\Delta$-regular, and for $\alpha\notin\{1, 2, \dots \}$ it
satisfies the singular expansion
\begin{equation}
\label{Li 1} Li_{\alpha,0}(z)\sim \Gamma(1-\alpha)t^{\alpha-1} +
\sum_{j\geq0}\frac{(-1)^j}{j!}\zeta(\alpha-j)t^j,
\end{equation}

where
\[t=-\log z =\sum_{l\geq 1} \frac{(1-z)^l}{l}.\]
For $r > 0$, the singular expansion of $Li_{\alpha,r}$ is obtained
 using formal derivations:
$$Li_{\alpha,r}(z)=(-1)^r \frac{\partial ^ r}{\partial \alpha ^r}Li_{\alpha,0}(z).$$
\end{lemma}

A natural consequence of this lemma (which is a particular case of
\cite[Lemme 2.6]{FK}), is that
\begin{equation}
\label{Li 2}Li_{\alpha,0}(z)=\Gamma(1-\alpha)(1-z)^{\alpha-1}
+O(|1-z|^{\alpha})+\zeta(\alpha)\BBone_{\alpha
>0};\hspace{0.4cm}\alpha < 1.
\end{equation}

\vspace{0.4cm} Another result, which is very useful in what
follows, is the decomposition of the Hadamard product of
$(1-z)^a\odot(1-z)^b$ (cf. \cite[Proposition 8]{FFK}).

\begin{lemma}
\label{decomposition}
 For the real numbers $a$ and $b$,
$$(1-z)^a\odot(1-z)^b\thicksim \sum_{k\geq 0}
{\lambda_k}^{(a,b)}\frac{(1-z)^k}{k!}+\sum_{k\geq
0}{\mu_k}^{(a,b)}\frac{(1-z)^{a+b+1+k}}{k!},$$ where the
coefficients $\lambda$ and $\mu$ are given by
$${\lambda_k}^{(a,b)}=\frac{\Gamma(1+a+b)}{\Gamma(1+a)\Gamma(1+b)}
\frac{(-a)^{\bar k}(-b)^{\bar k}}{(-a-b)^{\bar k}},$$

$${\mu_k}^{(a,b)}=\frac{\Gamma(-1-a-b)}{\Gamma(-a)\Gamma(-b)}
\frac{(1+a)^{\bar k}(1+b)^{\bar k}}{(2+a+b)^{\bar k}},$$ where
$x^{\bar k}$ is defined as $x(x+1)\dots(x+k-1)$, for $k$
nonnegative entire.
\end{lemma}

Now, equipped with the singularity analysis toolkit, we are in a
position to find the asymptotic average from the relation
(\ref{recHada1}).

\begin{lemma}
\label{mean} The expected value of the total cost, induced by the
toll $n^2$ in the model of random free trees defined in Section
\ref{cayrec}, is
\begin{equation}
\label{asmean2}a_n=\sqrt{\pi/8}\ n^{5/2}+O(n^{3/2}).
\end{equation}
\end{lemma}

\begin{proof}Since $b_n=n^2$, we have $B(z)=Li_{-2,0}(z)$ and
the equality (\ref{Li 2}) implies
\begin{equation}
\label{Bz} B(z)=2(1-z)^{-3}+O(|1-z|^{-2}).
\end{equation}
Considering the singular expansion (\ref{caygene}) of the
generating function of the tree, Lemma \ref{decomposition} gives
$$B(z)\odot C(z/e)^2=2^{-1/2}(1-z)^{-3/2}+O(|1-z|^{-1}).$$ Consequently,
\begin{eqnarray*}
\int_0^z \frac{\partial_\omega[B(\omega)\odot
C(\omega/e)^2]}{C(\omega/e)}d\omega &=&
\int_0^z\left[\frac{3(1-\omega)^{-5/2}}{2\sqrt2}+O(|1-\omega|^{-2})\right]d\omega
\\
&=&\frac{1}{\sqrt 2}\,(1-z)^{-3/2}+O(|1-z|^{-1}).
\end{eqnarray*}

 Finally by the relation
(\ref{recHada1}) we have
\begin{equation}
\label{expan A.C}A(z)\odot
C(z/e)=\frac{1}{4}(1-z)^{-2}+O(|1-z|^{-3/2}).
\end{equation}
Moreover, for $\alpha$ positive, we have (see \cite{FO}, for
example)
\begin{eqnarray}
\nonumber
[z^n](1-z)^{-\alpha}&=&{{n+\alpha-1}\choose{n}}\frac{\Gamma(n+\alpha)}{\Gamma(\alpha)\Gamma(n+1)}
\\
\label{gamma}&=&
\frac{n^{\alpha-1}}{\Gamma(\alpha)}\left(1+O(1/n)\right),
\end{eqnarray}
where $[z^n](1-z)^{-\alpha}$ denotes the $n$-th coefficient of
$z^n$ in the expansion of $(1-z)^{-\alpha}$ in entire series. The
last equality is obtained applying the Stirling formula. Then, by
the expansion of (\ref{expan A.C}) and singularity analysis, we
obtain
$$a_nc_ne^{-n}=\frac{n}{4\Gamma(2)}(1+O(1/n))+O(n^{1/2}).$$
Finally with $c_n=\frac{n^{-3/2}e^{n}}{\sqrt{2\pi}}(1+O(1/n))$, we
obtain (\ref{asmean2}).\end{proof}

\vspace{0.6cm} Now estimating the moments of higher order, we
return to the recurrence (\ref{rec1}). For $k\ge 0, n\geq 1$, put
\[\mu_n(k):=\mathbb E({X_n}^k),\]
and
\[{{\tilde{\mu}}_{n}}(k):=c_n\,e^{-n} {{\mu}_n}(k).\]
Let ${M_k}(z)$ denote the ordinary generating function of
${{\tilde{\mu}}_n}(k)$, with $z$ marking $n$. For $k=1$,
\[{{\tilde{\mu}}_{n}}(1):=c_n\,e^{-n} {a_n}\hspace{0,5cm}\text{ and }\hspace{0,5cm}M_{1}(z)=A(z)\odot C(z/e).\]
For $k \geq 2$, we have
$$
{X_{n}}^k=
\sum_{k_1+k_2+k_3=k}{{k}\choose{k_1,k_2,k_3}}{X^{k_1}_{L_{n}}}{X^{k_2}_{n-L_{n}}}{{b_n}^{k_3}},
$$
or again
$${X_{n}}^k = {X^{k}_{L_{n}}}+{X^{k}_{n-L_{n}}}+\sum_{k_1+k_2+k_3=k\atop{k_1,k_2<k}}{{k}
\choose{k_1,k_2,k_3}} {X^{k_1}_{L_{n}}}{X^{k_2}_{n-L_{n}}}
{{b_n}^{k_3}}.$$ Conditioning on the size of $L_{n}$, we obtain
\begin{eqnarray*}
{\mu_n}(k)
 &=&
\sum_{k_1+k_2+k_3=k\atop{k_1,k_2<k}}{{k}\choose{k_1,k_2,k_3}}n^{2k_3}
\sum_{j=1}^{n}\frac{n}{2(n-1)}\frac{c_j c_{n-j}}{c_n}\mu_j
(k_1)\mu_{n-j} (k_2)
\\
\lefteqn{+\sum_{j=1}^{n}\frac{n}{2(n-1)}\frac{c_j
c_{n-j}}{c_n}(\mu_j (k)+\mu_{n-j} (k)).}
\end{eqnarray*}

Multiplying the latter by $\frac{n-1}{ne^n}c_n$, we obtain
\begin{equation}
\label{recmu} \frac{n-1}{n}{{\tilde{\mu}}_{n}}(k)
=\sum_{j=1}^{n-1}{\frac{c_{n-j}}{e^{n-j}} {{\tilde{\mu}}_{j}}(k)}
+r_n(k),
\end{equation}
where
$$r_n(k)=\sum_{k_1+k_2+k_3=k\atop{k_1,k_2<k}}{{k}\choose{k_1,k_2,k_3}}
{b_n}^{k_3}\sum_{j=1}^{n-1}{\frac{1}{2}
{{\tilde{\mu}}_{j}}(k_1)}{{\tilde{\mu}}_{n-j}}(k_2).$$

Let ${R_k}(z)$ denote the ordinary generating function of
$r_n(k)$,  with $z$ marking $n$. Therefore
\begin{equation}
\label{R.gene}{R_k}(z)=\sum_{k_1+k_2+k_3=k\atop{k_1,k_2<k}}{{k}\choose{k_1,k_2,k_3}}
(B(z)^{\odot k_3})\odot[1/2{M_{k_1}}(z){M_{k_2}}(z)],
\end{equation}
where
$$B(z)^{\odot k_3}:=\underbrace{B(z)\odot\dots\odot B(z)}_{{k_3}\text{ time}}.$$
Multiplying (\ref{recmu}) by $z^n$ and summing over $n\geq 1$, we
obtain
$$M_k (z)=\int_0^z M_k (\omega) \frac{d\omega}{\omega}M_k (z)C(z/e)+R_k (z),$$
which is identified in the equality (\ref{differ}) if
  there we choose $f(z)=M_k (z)$ and $t(z)=R_k (z)$. Finally, the
solution of this equation is
\begin{equation}
\label{recHada2}{M_k}(z)=\frac{C(z/e)}{1-C(z/e)}\int_0^z
\partial_\omega {R_k}(\omega)\frac{d\omega}{C(\omega/e)}.
\end{equation}

\vspace{.4cm}
\begin{proposition}
\label{prop} For $k \geq 1$, the generating function ${M_k}(z)$ of
${{\tilde{\mu}}_n}(k)$  satisfies
\begin{equation}
\label{expan M}{M_k}(z) = \frac{\sqrt 2}{2}{A_k}\ (1 - z)^{-5k/2+
\frac{1}{2}} + O(|1 - z|^{-5k/2+ 1}),
\end{equation}
where the coefficients $A_k$ are defined by the recurrence
\begin{equation}
\label{recprop}A_k = \sum_{j=1}^{k-1}{{{k}\choose{j}}\frac{A_j
A_{k-j}}{2}}+kA_{k-1} \frac{\Gamma(5k/2-1)}{\Gamma(5k/2-3)},
\hspace{0,5cm} k\geq 2;\hspace{0,5cm}A_1=2^{-3/2}.
\end{equation}
\end{proposition}

\begin{proof}\
 The proof is carried out by induction. For $k = 1$, the proposition has been
 established in view of (\ref{expan A.C}).
For $k\geq 2$, we demonstrate that ${R_k}(z)$ has a singular
expansion in the form
\begin{equation}
\label{expanR}{R_k}(z)={A_k}(1 - z)^{-5k/2+ 1} + O(|1 - z|^{-5k/2+
\frac{3}{2} }).
\end{equation}
Analyzing the various terms on the right hand side of
(\ref{R.gene}), we observe that $A_k$ are defined by the recurrence (\ref{recprop}):\\
\begin{itemize}
\item[{(I)}] \ By induction hypothesis, when $k_1$ and $k_2$ are
both nonzero, and $k_3=0$, the contribution to $R_k(z)$ is
\begin{eqnarray*}
\frac{1}{2}{M_{k_1}}(z){M_{k_2}}(z)&=&\frac{1}{2}\left[{A_{k_1}}(1
- z)^{\frac{-5{k_1}}{2}+ \frac{1}{2}} + O(|1-
z|^{\frac{-5{k_1}}{2}+ 1})\right]
\\
& \times & \left[{A_{k_2}}(1 - z)^{\frac{-5{k_2}}{2}+ \frac{1}{2}}
+ O(|1 - z|^{\frac{-5{k_2}}{2}+ 1})\right]
\\
&=&\frac{1}{2}{A_{k_1}}{A_{k_2}}(1 - z)^{\frac{-5{k}}{2}+ 1} +
O(|1 - z|^{\frac{-5{k}}{2}+ 3/2}).
\end{eqnarray*}

\vspace{.4cm} \item[(II)] \ When $k_1$, $k_2$ and $k_3$ are all
nonzero, by relation (\ref{Li 2}) and the relation below
\begin{eqnarray*}
\frac{1}{2}{M_{k_1}}(z){M_{k_2}}(z)=\frac{{A_{k_1}}{A_{k_2}}}{2\Gamma(\frac{5(k_1
+k_2)}{2}-1)}Li_{\frac{-5k}{2}+\frac{5k_3}{2}+2,0}(z)+O(|1 -
z|^{\frac{-5(k_1+k_2)}{2}+ 3/2}),
\end{eqnarray*}
and since $B(z)^{\odot k_3}=Li_{-2 k_3,0}(z)$, the contribution to
$R_k(z)$ is
\begin{eqnarray*}
Li_{-2k_3,0}(z)\odot[\frac{1}{2}{M_{k_1}}(z){M_{k_2}}(z)]
&=&\frac{{A_{k_1}}{A_{k_2}}}{2\Gamma(\frac{5(k_1
+k_2)}{2}-1)}Li_{\frac{-5k}{2}+\frac{k_3}{2}+2,0}(z)
\\
&+& Li_{-2 k_3,0}(z)\odot O(|1 - z|^{\frac{-5(k_1+k_2)}{2}+ 3/2})
\\
&=& O(|1 - z|^{\frac{-5k}{2}+ 3/2}).
\end{eqnarray*}

\vspace{.4cm} \item[(III)]Consider now the case where $k_1$ is
nonzero and where $k_2=0$. We have $M_0(z)=C(z/e)$. The
contribution to $R_k(z)$ is the ${{k}\choose{k_1}}$ times
\begin{eqnarray*}
\frac{1}{2}{M_{k_1}}(z){M_{k_2}}(z)&=&\frac{1}{2}\left[{A_{k_1}}(1
- z)^{\frac{-5{k_1}}{2}+ \frac{1}{2}} + O(|1 -
z|^{\frac{-5{k_1}}{2}+ 1})\right]
\\
&\times&\left[1-\sqrt 2 (1 - z)^{\frac{1}{2}}+ O(|1 - z|)\right]
\\
&=&\frac{A_{k_1}}{2\Gamma(\frac{5 k_1 }{2}-1/2)}Li_{\frac{-5
k_1}{2}+\frac{3}{2},0}(z)+O(|1 - z|^{\frac{-5{k_1}}{2}+ 1}).
\end{eqnarray*}
Since
 \begin{eqnarray*}
Li_{-2k_3,0}(z)\odot[\frac{1}{2}{M_{k_1}}(z){M_{k_2}}(z)]
&=&\frac{A_{k_1}}{2\Gamma(\frac{5 k_1 }{2}-1/2)}Li_{\frac{-5
k}{2}+\frac{k_3}{2}+\frac{3}{2},0}(z)
\\
&+& Li_{-2 k_3,0}(z)\odot O(|1 - z|^{\frac{-5 k_1}{2}+1}),
\end{eqnarray*}the contribution to $R_k(z)$, for $k_3\geq 2$, is
$$O(|1 - z|^{\frac{-5k}{2}+ {k_3}/2+1/2})=O(|1 - z|^{\frac{-5k}{2}+
3/2}).$$

\vspace{.4cm}
 \item[(IV)]In the case where $k_1$ is nonzero,
 $k_2=0$ and $k_3=1$, the contribution to $R_k(z)$ is
${{k}\choose{k-1}}=k$ times
$$\frac{A_{k-1}\Gamma(\frac{5 k }{2}-1)}{2\Gamma(\frac{5 k }{2}-3)}=(1 -
z)^{\frac{-5 k}{2}+ 1}+ O(|1 - z|^{\frac{-5 k}{2}+ 3/2}).$$
\vspace{.2cm}
 \item[(V)] The case where $k_2$ is nonzero and $k_1 = 0$
is identical to two preceeding cases.

 \vspace{.4cm}
 \item[(VI)]The last
contribution comes from the single term when both $k_1$ and $k_2$
are zero. In this case, the contribution to $R_k(z)$ is
\begin{eqnarray*}
{B}(z)^{\odot k}\odot [\frac{1}{2} C(\frac{z}{e})^2] &=&
Li_{-2k,0}(z)\odot \left(1/2-\sqrt 2
(1-z)^{\frac{1}{2}}+O(|1-z|)\right)
\\
&=&  Li_{-2k,0}(z)\odot \left(-\frac{\sqrt 2}{\Gamma(-1/2)}
Li_{3/2,0}(z)+O(1)\right)
\\
&=& O(|1-z|^{-2k+3/2-1})=O(|1-z|^{-5k/2+3/2}).
\end{eqnarray*}
\end{itemize}
\vspace{.4cm} Adding all these six contributions yields the
expansion (\ref{expanR}), as well as the recurrence formula
(\ref{recprop}). Utilizing (\ref{expanR}) in (\ref{recHada2}), we
finally obtain the expansion (\ref{expan M}).
\end{proof}

 \vspace{0.2cm}
 \noindent
%%%%%%%%%%%%%%%%%%%%%%%%%%%%%%%%%%%%%%%%%%%%%%%%
\section{Proof of Theorem \ref{Xn}}
\label{limdis}
%%%%%%%%%%%%%%%%%%%%%%%%%%%%%%%%%%%%%%%%%%%%%%%%
According to Proposition \ref{prop}, the generating function
${M_k}(z)$ of $\left({c_n}e^{-n}{\mu}_n(k)\right)_{k\geq 1}$ has
the singular expansion
$${M_k}(z) = \frac{\sqrt 2}{2}{A_k}(1 - z)^{-5k/2+
\frac{1}{2}} + O(|1 - z|^{-5k/2+ 1}),$$ where $A_k$ satisfy the
recurrence (\ref{recprop}). Thus, having
\[\frac{c_n}{e^n}=\frac{n^{-3/2}}{\sqrt{2\pi}}(1+O(1/n)),\]
in view of (\ref{gamma}) and the techniques of singularity
analysis, we obtain
\begin{equation}
\label{mu}{\mu}_n(k)=\frac{A_k \sqrt \pi}{\Gamma(\frac{5k-1}{2})}\
n^{5k/2}+O(n^{5k/2-1/2}).
\end{equation}
We will utilize this estimate of the $k$-th moment to derive from
it the limit distribution of our additive functional. From
(\ref{mu}) we obtain, for $k\geq 1$,
\begin{equation}
\label{tense}\mathbb E\left[\left(n^{-5/2}\,X_n\right)^k\right]
=\frac{A_k \sqrt\pi}{\Gamma(\frac{5k-1}{2})}+O(n^{-1/2}).
\end{equation}
Once we prove the following
 lemma, the hypothesis of \cite[Theorem
30.1]{Bil}, is verified and we can be sure that the suite of
$\frac{A_k \sqrt\pi}{\Gamma(\frac{5k-1}{2})}$ characterizes a
unique probability law.
\begin{lemma}
\label{const1} There exist a constant $C<\infty $ such that
$$\left|\frac{A_k}{k!}\right|\leq C^k k^{5k/2},$$ for all
$k\geq 1$.
\end{lemma}

\begin{proof}
The demonstration is by induction. For $k\in\{1,2\}$, the
inequality is satisfied, if we choose the constant $C$
sufficiently large. For $k\geq 2$, putting $s_k:=\frac{A_k}{k!}$
and dividing the recurrence (\ref{recprop}) by $k!$, we obtain
\begin{eqnarray*}
s_k &=& \frac 1 2 \sum_{j=1}^{k-1}s_j s_{k-j} + s_{k-1}\ (5k/2
-2)(5k/2 -3)
\\
&\le& \frac 1 2 \sum_{j=1}^{k-1}s_j s_{k-j} + \gamma\ s_{k-1}\
k^2,
\end{eqnarray*}
for $\gamma=25/4$. By the induction hypothesis,
$$|s_k|\ \leq \ \frac {C^k}2\ \sum_{j=1}^{k-1}\ |j^j
(k-j)^{k-j}|^{5/2} + \gamma\ C^{k-1}\ (k-1)^{\frac{5(k-1)}2}\
k^2.$$ Since, for $0<j\leq k/2$, the term $j^j (k-j)^{k-j}$
decrease when $j$ grows, we can limit the sum, considering the sum
for $j=1$, $j=k-1$ and $k-2$ times $j=2$. Then, for $k\ge 3$,
\begin{eqnarray*}
 \left|s_k\right| &\leq& \frac {C^k}2
[(k-1)^{k-1}+2(k-2)^{k-1}]^{5/2}+ \gamma C^{k-1}
k^{\frac{5(k-1)}2}
\\
& \leq & \frac {C^k}2 (3k^{k-1})^{5/2}+C^k \frac{\gamma}C
k^{\frac{5(k-1)}2}
\\
& \leq & C^{k}\ k^{\frac{5k}2},
\end{eqnarray*}
where the last inequality justified when we choose $C\ge
2\gamma\,3^{-5/2}$. \end{proof}

\vspace{.3cm} It follows from Lemma \ref{const1} that, for $B$
sufficiently large,
\begin{equation}
\label{const2}\left|\frac{A_k
\sqrt\pi}{k!\Gamma(\frac{5k-1}{2})}\right|\leq B^k,
\end{equation}

and by \cite[Theorem 30.1]{Bil}, there exists a unique probability
distribution having the moments $\frac{A_k
\sqrt\pi}{k!\Gamma(\frac{5k-1}{2})}$. Let $Y$ be a random variable
having such a probability distribution. We deduce that
$$n^{-5/2}\,X_n\ \build{\longrightarrow }{}{\mathcal L}\ Y.$$
Putting $\xi=\frac{Y}{\sqrt 2}$ and ${\bar
a}_k=\frac{2^{3k-1}}{k!}A_k$, we obtain $$\mathbb
E(\xi^k)=\frac{k!\sqrt
\pi}{2^{(7k-2)/2}\Gamma(\frac{5k-1}{2})}{\bar a}_k,$$ and
$${\bar a}_k=2(5k-6)(5k-4){\bar a}_{k-1}+\sum_{j=1}^{k-1}{{\bar a}_j
{\bar a}_{k-j}}\hspace{0,5cm} k\geq 2;\hspace{0,5cm}{\bar
a}_1=\sqrt 2,$$ what is the statement of Theorem \ref{Xn}.

%%%%%%%%%%%%%%%%%%%%%%%%%%%%%%%%%%%%%%%%%%%%%%%%%%%%%%%%%%%%%%%%%%%%%%%%%%%%%%%%%%%%%%%

\section*{Acknowledgments}
I wish to thank Philippe Chassaing for the fruitful discussions
and the referee for his valuable comments that improved the
presentation of this article.

School of mathematics and computer sciences, Damghan university,
Iran, p.o.Box 36716-41167\\
 Email: zohorian@du.ac.ir


\begin{thebibliography}{10}

\bibitem{Bil}
  Billingsley, P. (1995). {\em Probability and measure}, John Wiley \&
Sons.

\bibitem{Cay} {Cayley, A.} (1889). {A theorem on
trees},  {\em Q. J. Pure Appl. Math.}, {23}, pp. 376-378.

\bibitem{CLR} {Cormen, T. H. \& Leiserson, C. E. and Rivest, R. L.} (1990).
{\em Introduction to algorithms}, {McGraw-Hill}.

\bibitem{FK} {Fill, J. A. and Kapur, N. } (2004).
{{ Limiting distributions of addititive functionals on Catalan
trees}}, {\em Theoret. Comput. Sci.}, {326 (1-3)}, pp. 69-102.

\bibitem{FFK}{Fill, J. A. and Flajolet, P. and Kapur, N.} (2005).
{{ Singularity analysis, Hadamard products, and Tree
recurrences}}, {\em{ Journal of Computational and Applied
Mathematics}}, 174 (2), pp. 271-313.

\bibitem{FO} {Flajolet, P. and Odlyzko, A.} (1990).
{{ Singularity analysis of generating functions}}, {\em{SIAM J.
Discrete Math.}},  {3 (2)}, {pp. 216-240}.

\bibitem{Hwang} {Hwang, H. K. and Neininger, R. } (2002).
{{ Phase change of limit laws in the quicksort recurrence under
varying toll functions}}, {\em{SIAM J. Comput.}},  {31 (6)}, {pp.
1687-1722}.

\bibitem{Jan}{Janson, S. } (2003).  {{ The Wiener index
of simply generated random trees}},  {\em{Random Structures
Algorithms}},  {22 (4)}, {pp. 337-358}.

\bibitem{KP} {Knuth, D. E. and Pittel, B.} (1989). {{ A
recurrence related to trees}}, {\em{Proceedings of the American
Mathematical Society}}, {105 (2)}, {pp. 335-349}.

\bibitem{KS}{Knuth, D. E.  and  Sch{\"o}nhage, A.} (1978).
 The expected linearity of a simple equivalence algorithm,
 {\em{Theoret. Comput. Sci.}}, {6, no. 3}
{pp. 281-315}.

\bibitem{Mahmoud} {Mahmoud, H. M.} (1992).
{{\em Evolution of random search trees}}, {John Wiley \& Sons
Inc}.

%\bibitem{maj}{Majumdar, S. N. and Dean, D. S.} (2002). { Exact Solution of a Drop--push Model for Percolation},
%{\em {Phy Rev Ltt}}, 89(11), pp. 115701.

\bibitem{Rosler}{R\"{o}sler, U.} (1991). { A limit
theorem for Quicksort}, {\em RAIRO Inform. Théor. Appl.}, {25 (1)}
, {pp. 85-100}.

\bibitem{Se}{ Sedgewick, R.} (1988).
{{\em Algorithms}}, {Addison--Wesley}.

\bibitem{VF} { Vitter, J. S. and Flajolet, P.} (1990).
{{ Analysis of algorithms and data structures}}, {\em{In Handbook
of Theoretical Computer Science, J. van Leeuwen, Algorithms and
Complexity. North Holland}}, {A ch. 9}, {pp. 431-524}.

\bibitem{EZoh} {Zohoorian A., E.} (2010). Analysis of a Drop-Push model
for Percolation and Coagulation, \em{Accepted to publish by the
journal of Statistical Physics}.
\end{thebibliography}
\end{document}